\documentclass[10pt,oneside,english]{amsart}
\usepackage[T1]{fontenc}
\usepackage[latin9]{inputenc}
\usepackage[a4paper]{geometry}
\usepackage{amsthm}
\usepackage{amssymb}

\makeatletter
\numberwithin{equation}{section}
\numberwithin{figure}{section}
  \theoremstyle{plain}
  \newtheorem*{thm*}{Theorem}
\theoremstyle{plain}
\newtheorem{thm}{Theorem}
  \theoremstyle{plain}
  \newtheorem{prop}[thm]{Proposition}
  \theoremstyle{plain}
  \newtheorem{cor}[thm]{Corollary}
 \theoremstyle{definition}
  \newtheorem{example}[thm]{Example}
  \theoremstyle{remark}
  \newtheorem{rem}[thm]{Remark}

\theoremstyle{definition}
\newtheorem{parn}{}[section]

\makeatother

\usepackage{babel}

\begin{document}

\title{Locally tame plane polynomial automorphisms }

 \author{Joost Berson}
  \thanks{Funded by a free competition grant of the Netherlands Organisation for scientific research (NWO)}
 \address{Radboud University Nijmegen \\ Postbus 9010 \\ 6500 GL Nijmegen \\ The Netherlands } 
 \email{j.berson@sience.ru.nl}

\author{Adrien Dubouloz}
 \thanks{Partially supported by FABER grant 07-512-AA-010-S-179 and PHC Grant Van Gogh 18153NA}
\address{Institut de Math\'ematiques de Bourgogne \\ Universit\'e de Bourgogne \\9 avenue Alain Savary \\ BP 47870 \\ 21078 Dijon cedex \\ France} 
\email{Adrien.Dubouloz@u-bourgogne.fr}

\author{ Jean-Philippe Furter}
\thanks{Partially supported by PHC Grant Van Gogh 18153NA}
\address{Universit\'e de La Rochelle \\ Avenue Michel Cr\'epeau \\ 17000 La Rochelle, France}
\email{jpfurter@univ-lr.fr}

\author{ Stefan Maubach} 
   \thanks{Funded by Veni-grant of council for the physical sciences, Netherlands Organisation for scientific research (NWO)}
%\address{Radboud University Nijmegen \\ Postbus 9010 \\ 6500 GL Nijmegen \\ The Netherlands }  
\address{Jacobs University Bremen, Campus Ring 1, 28759 Bremen, Germany}
%%\email{s.maubach@science.ru.nl}
\email{s.maubach@jacobs-university.de}

\begin{abstract}
For automorphisms of a polynomial ring in two variables over a domain
$R$, we show that local tameness implies global tameness provided
that every $2$-generated locally free $R$-module of rank $1$ is
free. We give many examples illustrating this property. 
\end{abstract}
\maketitle

\section*{Introduction }

A natural problem in commutative algebra and algebraic geometry is
to understand the group $\mbox{{\rm GA}}_{n}\left(R\right)$ of automorphisms
of a polynomial ring $R\left[X_{1},\ldots,X_{n}\right]$ over a ring
$R$. Although much progress has been made in this direction during
the last decades, one can state that only the case $n=2$ and $R$
is a field is fully understood. A central and fruitful notion in the
study of polynomial automorphisms is the notion of tameness: an automorphism
is called \emph{tame} if it can be written as a composition of affine
and triangular ones, where by a triangular automorphism, we mean an
automorphism $F=\left(F_{1},\ldots,F_{n}\right)\in\mbox{{\rm GA}}_{n}\left(R\right)$
such that $F_{i}\in R\left[X_{i},\ldots,X_{n}\right]$ for every $i=1,\ldots,n$.
Tame automorphisms form a subgroup ${\rm TA}_{n}\left(R\right)$ of
${\rm GA}_{n}\left(R\right)$ and a classical theorem due to Jung in
characteristic zero \cite{Jung} and van der Kulk in the general case
\cite{Kulk} asserts that if $R$ is a field $k$ then ${\rm GA}_{2}\left(k\right)={\rm TA}_{2}\left(k\right)$. 
The result is even more precise: ${\rm GA}_{2}\left(k\right)$ is the free product 
of the subgroups of affine and triangular automorphisms amalgamated over their intersection. 
In contrast, even the equality ${\rm TA}_{2}\left(R\right)={\rm GA}_{2}\left(R\right)$
is no longer true for a general domain $R$, as illustrated by a 
famous example due to Nagata : for an element $z\in R\setminus\{0\}$ the endomorphism 
\[F=\left(X-2Y(zx+Y^{2})-z(zX+Y^{2})^{2},Y+z(zX+Y^{2})\right)\]
of $R\left[X,Y\right]$ is in ${\rm GA}_{2}\left(R\right)$ and can be decomposed as
$F=(X-z^{-1}Y^2,Y)(X,z^2X+Y)(X+z^{-1}Y^2,Y)$ in 
${\rm GA}_{2}\left(K(R)\right)={\rm TA}_{2}\left(K(R)\right)$. 
Such a decomposition being essentially unique, this implies in particular that if $z$ is not invertible in $R$, 
then $F$ cannot be tame over $R$. Note that more generally, given a prime ideal $\mathfrak{p}\in{\rm Spec}\left(R\right)$, $F\in{\rm TA}_{2}\left(R_{\mathfrak{p}}\right)$ if and only if $z\not\in\mathfrak{p}$. 

Automorphisms $F\in{\rm GA}_{n}\left(R\right)$ such that $F\in{\rm TA}_{n}\left(R_{\mathfrak{p}}\right)$
for every $\mathfrak{p}\in{\rm Spec}\left(R\right)$ are said to be
\emph{locally tame}. Of course, every tame automorphism is locally
tame. In contrast, the Nagata automorphism is neither tame nor locally
tame. This could suggest that, at least for plane polynomial automorphisms,
tameness is a property that can be checked locally on the base ring.
In particular, one could hope that the only reason why an automorphism
$F\in{\rm GA}_{2}\left(R\right)$ is not tame is because there exists
a prime $\mathfrak{p}\in{\rm Spec}\left(R\right)$ such that $F$
is already non tame over $R_{\mathfrak{p}}$. It turns out that this
hope is too optimistic, and that in general, some ``global'' properties
of $R$ have to be taken into account to be able to infer tameness
directly from local tameness. The main result of this article is the
following characterization of rings for which global tameness can
be checked locally : 
\begin{thm*}
For a domain $R$, the following assertions are equivalent :

1) ${\displaystyle {\rm TA}_{2}\left(R\right)=\bigcap_{\mathfrak{p}\in{\rm Spec}\left(R\right)}}{\rm TA}_{2}\left(R_{\mathfrak{p}}\right)$,

2) Every $2$-generated locally free $R$-module of rank $1$ is free. 
\end{thm*}
In particular, it follows that over a unique factorization domain
$R$, tameness is a local property of automorphisms. \\

The article is organized as follows. Section one is devoted to the
proof of the above characterization, that we essentially derive from
the fact that tame automorphisms of a polynomial ring in two variables
can be recognized algorithmically. In section two, we consider many
examples that illustrate condition 2) in the Theorem above.

\section{From Local tameness to global tameness}

In this section, we characterize domains $R$ with the property that
an automorphism $F=\left(F_{1},F_{2}\right)\in{\rm GA}_{2}\left(R\right)$
is tame if and only if it is locally tame. 

\begin{parn} \textbf{Notations.} For an automorphism $F=\left(F_{1},F_{2}\right)\in{\rm GA}_{2}\left(R\right)$,
we let $\deg F=\left(\deg F_{1},\deg F_{2}\right)\in(\mathbb{N}^*)^{2}$
considered as equipped with the product order. We denote by $\overline{F_{i}}$
the homogeneous component of $F_{i}$ of degree $\deg F_{i}$, $i=1,2$.
An automorphism with $\deg F=\left(1,1\right)$ is affine, and we
denote by ${\rm Aff}_{2}\left(R\right)$ the corresponding subgroup
of ${\rm GA}_{2}\left(R\right)$. 

\end{parn}

\subsection{Properties of automorphisms}

\indent\newline\indent Even if the equality ${\rm GA}_{2}\left(R\right)={\rm TA}_{2}\left(R\right)$
is no longer true for a general domain $R$, tame automorphisms of
a polynomial ring in two variables can be recognized algorithmically.
Indeed, the following result quoted from \cite[Prop. 1]{Fur} (see
also \cite[Cor. 5.1.6]{Essenbook}) says in essence that for every
$F\in{\rm TA}_{2}\left(R\right)$ with $\deg F>\left(1,1\right)$
there exists a linear or a triangular automorphism $\varphi$ such that
$\deg\varphi F<\deg F$. 
\begin{prop}
\label{pro:Tame-algo}Let $F=\left(F_{1},F_{2}\right)\in{\rm TA}_{2}\left(R\right)$
and let $\left(d_{1},d_{2}\right)=\deg F$. Then the following holds:

a) $d_{1}\mid d_{2}$ or $d_{2}\mid d_{1}$.

b) If $\max\left(d_{1},d_{2}\right)>1$ then we have: 

$\qquad$ (i) If $d_{1}<d_{2}$ then $\overline{F_{2}}=c\overline{F_{1}}^{d_{2}/d_{1}}$
for some $c\in R$,

$\qquad$ (ii) If $d_{2}<d_{1}$ then $\overline{F_{1}}=c\overline{F_{2}}^{d_{1}/d_{2}}$
for some $c\in R$,

$\qquad$ (iii) If $d_{1}=d_{2}$ then there exists $\varphi\in{\rm Aff}_{2}(R)$
such that $\varphi F=\left(F_{1}',F_{2}'\right)$ satisfies $\deg F_{1}'=d_{1}$
and $\deg F_{2}'<d_{1}$. 
\end{prop}
\begin{parn} \label{par:loc-free-obs}In contrast to the tame case,
for an arbitrary automorphism $F=\left(F_{1},F_{2}\right)\in{\rm GA}_{2}\left(R\right)$
with $\deg F_{1}=\deg F_{2}$ there is no guarantee in general that
there exists $\varphi\in{\rm Aff}_{2}\left(R\right)$ such that $\deg\varphi F<\deg F$.
Indeed, such a $\varphi$ exists if and only there exists a unimodular
vector $\left(\alpha_{1},\alpha_{2}\right)\in R^{2}$ such that $\alpha_{1}\overline{F_{1}}+\alpha_{2}\overline{F_{2}}=0$,
which is the case if and only if the $R$-module $R\overline{F_{1}}+R\overline{F_{2}}$
is free of rank $1$. Combined with \cite[Ex. 6 p. 94]{Essenbook},
this observation leads to a natural procedure to construct families
of locally tame but not (globally) tame automorphisms, namely: 

\end{parn}
\begin{prop}
\label{pro:Can-Ex}If $z,w\in R$ and $q\left(T\right)\in R\left[T\right]$
is a polynomial of degree at least $2$, then \begin{eqnarray*}
F & := & \left(X+wq\left(zX+wY\right),Y-zq\left(zX+wY\right)\right)\end{eqnarray*}
is an element of ${\rm GA}_{2}\left(R\right)$. Furthermore $F$ is
tame if and only if $\left(z,w\right)$ is a principal ideal of $R$. 

In particular, if $\left(z,w\right)$ is a locally principal but not
principal ideal, then $F$ is a locally tame but not globally tame
automorphism. \end{prop}
\begin{proof}
A straightforward verification shows that $G=\left(X-wq\left(zX+wY\right),Y+zq\left(zX+wY\right)\right)$
is an inverse for $F$. 
Suppose that $\left(z,w\right)=aR$ for some $a\in R$. Replacing
$q\left(T\right)$, $z$ and $w$ by $aq\left(aT\right),a^{-1}z$
and $a^{-1}w$ respectively, we may assume that $\left(z,w\right)=R$.
But then if we take any $\varphi\in{\rm SL}_{2}\left(R\right)$ having
$zX+wY$ as its first component, one checks that $F=\varphi^{-1}\left(X,Y-q\left(X\right)\right)\varphi\in TA_{2}\left(R\right)$.
Conversely, if $F\in{\rm TA}_{2}\left(R\right)$, then, since $\deg F_{1}=\deg F_{2}=\deg Q>1$,
it follows from Proposition \ref{pro:Tame-algo} and the above discussion
that the $R$-module generated by $\overline{F_{1}}=w\overline{q\left(zX+wY\right)}$
and $\overline{F_{2}}=-z\overline{q\left(zX+wY\right)}$ is free of
rank $1$. Simplifying by $\overline{q\left(zX+wY\right)}$, we get
that the $R$-module generated by $w$ and $-z$ is free a rank $1$,
i.e., $\left(w,-z\right)$ is a principal ideal. 
\end{proof}
\begin{parn} It follows that locally tame but not globally tame automorphisms
abound : for instance, in the proposition above, one can take for $R$
the coordinate ring of a smooth non rational affine curve $C$ and
for $z,w$ a pair of generators of the defining ideal of a non principal
Weil divisor on $C$ (see also section 2 below for more examples). 

\end{parn}

\subsection{A criterion}

\indent\newline\indent It turns out that the examples discussed above
illustrate the only global obstruction to infer global tameness from
local tameness, namely, the existence of $2$-generated locally free
but not globally free modules of rank $1$. Indeed, we have the following
criterion.
\begin{thm}
\label{thm:Main}For a domain $R$, the following assertions are equivalent:

1) ${\displaystyle {\rm TA}_{2}\left(R\right)=\bigcap_{\mathfrak{p}\in{\rm Spec}\left(R\right)}}{\rm TA}_{2}\left(R_{\mathfrak{p}}\right)$,

2) Every $2$-generated locally free $R$-module of rank $1$ is free. 
\begin{proof}
1) $\Rightarrow$ 2). Since $R$ is a domain, every locally free $R$-module
of rank $1$ is isomorphic to an $R$-submodule of the field of fractions
$K\left(R\right)$ of $R$ (see e.g. \cite[Prop. 6.15]{Ha}). In turn,
every such submodule is isomorphic to an ideal of $R$. In particular,
if there exists a locally free but non free $2$-generated $R$-module
of rank $1$, then there exists locally principal but not principal
ideal $\left(z,w\right)$ of $R$. But then any $F\in{\rm GA}_{2}\left(R\right)$
as in Proposition \ref{pro:Can-Ex} above is locally tame but not
tame. 

2) $\Rightarrow$ 1). Conversely, for any domain $R$, it is clear that \[
{\rm TA}_{2}\left(R\right)\subseteq\bigcap_{\mathfrak{p}\in{\rm Spec}\left(R\right)}{\rm TA}_{2}\left(R_{\mathfrak{p}}\right)\subseteq\bigcap_{\mathfrak{p}\in{\rm Spec}\left(R\right)}{\rm GA}_{2}\left(R_{\mathfrak{p}}\right)={\rm GA}_{2}\left(R\right).\]
 Let $F=\left(F_{1},F_{2}\right)\in{\rm GA}_{2}\left(R\right)$ be
a locally tame automorphism and let $d_{i}=\deg F_{i}$, $i=1,2$.
We may assume that $d_{1}\leq d_{2}$. If $d_{1}=d_{2}=1$ then $F$
is affine, whence tame. We now proceed by induction on $\left(d_{1},d_{2}\right)$,
assuming that every locally tame automorphism of degree $\left(d_{1}',d_{2}'\right)<\left(d_{1},d_{2}\right)$
is globally tame. 

$\bullet$ Case $1$ : $d_{1}<d_{2}$. Since $F\in{\rm TA}_{2}\left(R_{\left(0\right)}\right)={\rm TA}_{2}\left(K\left(R\right)\right)$,
it follows from Proposition \ref{pro:Tame-algo} that $e=d_{2}/d_{1}\in\mathbb{N}^{*}$
and that there exists $\alpha\in K\left(R\right)$ such that $\overline{F_{2}}=\alpha\overline{F_{1}}^{e}$.
But since $F\in{\rm TA}_{2}\left(R_{\mathfrak{p}}\right)$ for every
$\mathfrak{p}\in{\rm Spec}\left(R\right)$, it follows that \[
\alpha\in{\displaystyle \bigcap_{\mathfrak{p}\in{\rm Spec}\left(R\right)}R_{\mathfrak{p}}=R}.\]
 Now, the automorphism $\left(X,Y-\alpha X^{e}\right)F$ satisfies the induction hypothesis
and we are done with case.

$\bullet$ Case $2$ : $d_{1}=d_{2}$. Since for any $\mathfrak{p}\in{\rm Spec}\left(R\right)$,
we have $F\in{\rm TA}_{2}\left(R_{\mathfrak{p}}\right)$, it follows
from Proposition \ref{pro:Tame-algo} and the discussion \ref{par:loc-free-obs}
that for every $\mathfrak{p}\in{\rm Spec}\left(R\right)$, the $R_{\mathfrak{p}}$
module generated by $\overline{F_{1}}$ and $\overline{F_{2}}$ is
free of rank $1$. This means exactly that the $R$-module generated
by $\overline{F_{1}}$ and $\overline{F_{2}}$ is locally free of
rank $1$. Our assumption implies that it is globally free, and so,
we deduce from \ref{par:loc-free-obs} that there exist $\varphi\in{\rm Aff}_{2}\left(R\right)$
such that $\deg\varphi F<\deg F$. 
\end{proof}
\end{thm}
\begin{parn} Recall that the \emph{Picard group} of a ring $R$ is
the group ${\rm Pic}\left(R\right)$ of isomorphy classes of locally
free $R$-modules of rank $1$. In view of the above criterion, it
is natural to introduce the subgroup ${\rm Pic}_{2}\left(R\right)$
of ${\rm Pic}\left(R\right)$ generated by isomorphy classes of locally
free $R$-modules of rank $1$ that can be generated by $2$
elements. With this definition, property 2) in Theorem \ref{thm:Main}
is equivalent to the triviality of ${\rm Pic}_{2}\left(R\right)$.
In particular, we obtain: 

\end{parn}
\begin{cor}
If ${\rm Pic}_{2}\left(R\right)=\left\{ 1\right\} $ and $F$ belongs
to ${\rm GA}_{2}\left(R\right)$, then $F$ is tame if and only if
it is locally tame. \end{cor}
\begin{example}
The class of rings with ${\rm Pic}_{2}\left(R\right)=\left\{ 1\right\} $
contains in particular unique factorization domains since for these
domains the Picard group itself is trivial. This also holds for \emph{B\'ezout rings}, 
that is, domains in which every finitely generated ideal is principal (see e.g. \cite{Cohn}). 
\end{example}

\subsection{Minimal overring for tameness}

\indent\newline\indent Recall that ${\rm GA}_{2}\left(R\right)={\rm TA}_{2}\left(R\right)$
if and only if $R$ is a field \cite[Proposition 5.1.9]{Essenbook}.
If $F\in{\rm GA}_{2}\left(R\right)$, then $F$ is tame over the field
of fractions $K$ of $R$, but, in general, there does not exist a
smallest ring $S$ between $R$ and $K$ such that $F$ is tame over
$S$. Indeed, letting $R=\mathbb{C}\left[z,w\right]$ every automorphism
$F$ as in Proposition \ref{pro:Can-Ex} is tame over $R\left[z^{-1}\right]$
and $R\left[w^{-1}\right]$ but not over $R=R\left[z^{-1}\right]\cap R\left[w^{-1}\right]$.
However, if we further assume that $R$ is a B\'ezout domain, we
have the following result. 
\begin{prop}
Let $R$ be a B\'ezout domain and let $\left(R_{j}\right)_{j\in J}$
be a family of rings between $R$ and $K$ such that $R={\displaystyle \bigcap_{j\in J}}R_{j}$.
Then ${\rm TA}_{2}\left(R\right)={\displaystyle \bigcap_{j\in J}}{\rm TA}_{2}\left(R_{j}\right)$. \end{prop}
\begin{proof}
Similarly as in the proof of Theorem \ref{thm:Main}, we proceed by
induction on the degree of $F=\left(F_{1},F_{2}\right)\in{\rm GA}_{2}\left(R\right)\cap{\displaystyle \bigcap_{j\in J}}{\rm TA}_{2}\left(R_{j}\right)$,
the case $\deg F=\left(1,1\right)$ being obvious. Letting $d_{i}=\deg F_{i}$,
we may assume that $d_{1}\leq d_{2}$. 

$\bullet$ Case $1$ : $d_{1}<d_{2}$. Then $e=d_{2}/d_{1}\in\mathbb{N}^{*}$
and there exists $\alpha\in K$ such that $\overline{F_{2}}=\alpha\overline{F_{1}}^{e}$.
Since $F\in TA_{2}\left(R_{j}\right)$, we have $\alpha\in R_{j}$
for every $j\in J$, and so $\alpha\in R={\displaystyle \bigcap_{j\in J}}R_{j}$.
Now the automorphism $\left(X,Y-\alpha X^{e}\right)F$ satisfies the
induction hypothesis. 

$\bullet$ Case 2 : $d_{1}=d_{2}$. Since $\overline{F_{1}}$ and
$\overline{F_{2}}$ are $K$-linearly dependent, the $R$-module $R\overline{F_{1}}+R\overline{F_{2}}$
is isomorphic to a proper ideal of $R$. As $R$ is a B\'ezout domain,
the latter is free of rank $1$, and so, we conclude from \ref{par:loc-free-obs}
above that there exists $\varphi\in{\rm Aff}_{2}\left(R\right)$ such
that $\deg\varphi F<\deg F$. \end{proof}
\begin{prop}
If $R$ is a B\'ezout domain and $F\in{\rm GA}_{2}\left(R\right)$
then there exists a smallest ring $S$ between $R$ and $K\left(R\right)$
such that $F\in{\rm TA}_{2}\left(S\right)$. Furthermore, $S$ is
a finitely generated $R$-algebra.

If we assume further that $R$ is a principal ideal domain, then there
exists $r\in R\setminus\left\{ 0\right\} $ such that $S=R\left[r^{-1}\right]$. \end{prop}
\begin{proof}
Any ring between $R$ and $K\left(R\right)$ is again a B\'ezout
domain \cite[Theorem 1.3]{Cohn}. Therefore, the existence of $S$
is a consequence of the previous proposition. The fact that $S$ is
finitely generated follows from Proposition \ref{pro:Tame-algo} by
easy induction. For the last assertion, since $S$ is finitely generated
over $R$, there exists a finitely generated ideal $I\subset R$ and
an element $r\in R\setminus\left\{ 0\right\} $ such that $S=R\left[I/r\right]=\left\{ a/r^{k}\in K\left(R\right),a\in I^{k},k=0,1\ldots\right\} $.
Since $R$ is a p.i.d, $I$ is a principal ideal, say generated by
an element $g\in R$. After eliminating common factors if any, we
may assume that $r$ and $g$ are relatively prime and that $S=R\left[g/r\right]\subset R\left[r^{-1}\right]$.
But by B\'ezout identity, there exists $u,v\in R$ such that $ur+vg=1$
and so, $S=R\left[r^{-1}\right]$. \end{proof}
\begin{example}
If $F\in{\rm GA}_{2}\left(\mathbb{C}\left[z\right]\right)$, then
there exists a smallest ring $S$ between $\mathbb{C}\left[z\right]$
and $\mathbb{C}\left(z\right)$ of the form $\mathbb{C}\left[z\right]\left[r^{-1}\right]$
such that $F\in{\rm GA}_{2}\left(S\right)$. 
\end{example}

\section{Examples and complements}

Here we discuss examples of domains $R$ which illustrate the property
${\rm Pic}_{2}\left(R\right)=\left\{ 1\right\} $.

\subsection{The condition  $\mathbf{{\rm \bf{Pic}}_2(R)=\{1\}}$ for $1$-dimensional
noetherian domains}

\indent\newline\noindent If $R$ is a noetherian domain of Krull
dimension $1$, every locally free $R$-module of rank $j$ is generated
by at most $j+1$ elements (see e.g. \cite[Th. 5.7]{Matsumura}).
In particular, we have ${\rm Pic}\left(R\right)={\rm Pic}_{2}\left(R\right)$
for every noetherian domain of dimension $1$. As a consequence, we
get: 
\begin{example}
If $R$ is a Dedekind domain, the following are equivalent :

(1) ${\rm Pic}_{2}\left(R\right)=\left\{ 1\right\} $; $\qquad$ (2)
${\rm Pic}\left(R\right)=\left\{ 1\right\} $; $\qquad$ (3) $R$
is a UFD; $\qquad$ (4) $R$ is a p.i.d. 
\end{example}
For the coordinate ring $R$ of an affine curve $C$ defined over
an algebraically closed field, we have the following classical result: 
\begin{prop}
\label{pro:Rat-Curve}The Picard group of $R$ is trivial if and only
if $C$ is a nonsingular rational curve. \end{prop}
\begin{proof}
Let $\tilde{C}={\rm Spec}(\tilde{R})$ be the normalization of $C$. 
By virtue of \cite[Theorem 3.2]{Wiegand}, the natural surjection 
${\rm Pic}(C)\rightarrow {\rm Pic}(\tilde{C})$ is an isomorphism if and only if $R=\tilde{R}$.
Therefore, if ${\rm Pic(C)}$ is trivial, then $C$ is necessarily a nonsingular curve. 
Now it is well known that a nonsingular curve has trivial Picard group if and only if
it is rational (see e.g. \cite[§11.4 p. 261]{Ei}).  \end{proof}
\begin{cor}
Let $R$ be the coordinate ring of a rational affine curve and let
$\tilde{R}$ be its integral closure in $K\left(R\right)$. If $F\in{\rm GA}_{2}\left(R\right)$
is locally tame, then $F\in{\rm TA}_{2}\left(\tilde{R}\right)$. \end{cor}
\begin{proof}
Indeed, with the notation of the previous proof, one has $F\in{\rm TA}_{2}\left(\mathcal{O}_{p}\right)$
for every $p\in C={\rm Spec}\left(R\right)$ and so $F\in{\rm TA}_{2}\left(\tilde{\mathcal{O}}_{p}\right)$
for every $p\in C$. Since $\tilde{R}=\bigcap_{p\in C}\tilde{\mathcal{O}}_{p}$,
it follows that $F$ is locally tame over $\tilde{R}$, whence tame
by virtue of Proposition \ref{pro:Rat-Curve}. \end{proof}
\begin{example}
Let $R=\mathbb{C}\left[u,v\right]/\left(v^{2}-u^{3}\right)$ be the
coordinate ring of a cuspidal rational curve $C$. Via the homomorphism
$\mathbb{C}\left[u,v\right]\rightarrow\mathbb{C}\left[t\right]$,
$\left(u,v\right)\mapsto\left(t^{2},t^{3}\right)$ we may identify
$R$ with the subring $\mathbb{C}\left[t^{2},t^{3}\right]$ of $\mathbb{C}\left[t\right]$
and the integral closure $\tilde{R}$ of $R$ with $\mathbb{C}\left[t\right]$.
For every $a\in\mathbb{C}^{*}$, we let $I_{a}=\left(t^{2}-a^{2},t^{3}-a^{3}\right)$
be the maximal ideal of the smooth point $\left(a^{2},a^{3}\right)$
of $C$. In particular, $I_{a}$ is locally principal but one checks
easily that it is not principal. So for $\left(z,w\right)=\left(t^{2}-a^{2},t^{3}-a^{3}\right)$,
any automorphism $F$ as in Proposition \ref{pro:Can-Ex} is locally
tame but not tame. On the other hand, $I_{a}\tilde{R}$ is principal, generated by $t-a$, 
and so, $F\in{\rm TA}_{2}\left(\mathbb{C}\left[t\right]\right)$. 
\end{example}

\subsection{Examples of rings with $\mathbf{{\rm \bf{Pic}}_2(R)=\{1\}}$ but $\mathbf{{\rm \bf{Pic}}(R)\neq \{1\}}$ }

\indent\newline\indent As observed above, for $1$-dimensional domains
$R$, the triviality of ${\rm Pic}_{2}\left(R\right)$ is equivalent
to the one of ${\rm Pic}\left(R\right)$. Here we give examples of
domains with ${\rm Pic}_{2}\left(R\right)=\left\{ 1\right\} $ and
${\rm Pic}\left(R\right)\neq\left\{ 1\right\} $ which are coordinate
rings of smooth affine algebraic varieties. 

\begin{parn} Let $Q$ be a smooth quadric in the complex projective
space $\mathbb{P}^{n}=\mathbb{P}^{n}_{\mathbb{C}}$, $n\geq2$, and let $U=\mathbb{P}^{n}\setminus Q$.
As is well-known, $U$ is smooth affine variety with Picard group
isomorphic to $\mathbb{Z}_{2}$, generated by the restriction to $U$
of the invertible sheaf $\mathcal{O}_{\mathbb{P}^{n}}\left(1\right)$
on $\mathbb{P}^{n}$. Letting $R_{n}=\Gamma\left(U,\mathcal{O}_{U}\right)$
and $M_{n}=\Gamma\left(U,\mathcal{O}_{\mathbb{P}^{1}}\left(1\right)\right)$,
which is a locally free $R_n$-module of rank $1$, we have the following
result. 

\end{parn}
\begin{prop}
The minimal number of generators of $M_{n}$ as an $R_{n}$-module
is $\left[n/2\right]+1$. In particular, if $n\geq4$ then ${\rm Pic}_{2}\left(R_{n}\right)=\left\{ 1\right\} $
whereas ${\rm Pic}\left(R_{n}\right)\simeq\mathbb{Z}_{2}$. \end{prop}
\begin{proof}
Up to the action of ${\rm PGL}_{n+1}\left(\mathbb{C}\right)$, we
may assume that $Q\subset\mathbb{P}^{n}={\rm Proj}\left(\mathbb{C}\left[x_{0},\ldots,x_{n}\right]\right)$
is the hypersurface $q=0$, where $q=x_{0}^{2}+\cdots+x_{n}^{2}\in\mathbb{C}\left[x_{0},\ldots,x_{n}\right]$.
Letting $\mathcal{Q}\subset\mathbb{A}^{n+1}={\rm Spec}\left(\mathbb{C}\left[x_{0},\ldots,x_{n}\right]\right)$
be the quadric defined by the equation $q=1$, the natural map $\mathbb{A}^{n+1}\setminus\left\{ 0\right\} \rightarrow\mathbb{P}^{n}$
restricts to an \'etale  double cover $\mathcal{Q}\rightarrow\mathbb{P}^{n}\setminus Q$
expressing the coordinate ring $R_{n}$ of $\mathbb{P}^{n}\setminus Q$
as the ring of invariant functions of $A=\mathbb{C}\left[x_{0},\ldots,x_{n}\right]/\left(q-1\right)$
for the $\mathbb{Z}_{2}$-action induced by ${\rm -id}$ on $\mathbb{A}^{n+1}$.
With this description, $\mathcal{O}_{\mathbb{P}^{n}}\left(1\right)\mid_{U}$
coincides with the trivial line bundle $\mathcal{Q}\times\mathbb{A}^{1}$
equipped with the nontrivial $\mathbb{Z}_{2}$-linearization $\mathcal{Q}\times\mathbb{A}^{1}\ni\left(x,u\right)\mapsto\left(-x,-u\right)\in\mathcal{Q}\times\mathbb{A}^{1}$
(see e.g. \cite[§1.3]{Mumford}). It follows that we may identify
regular functions on $U$ and global sections of $\mathcal{O}_{\mathbb{P}^{n}}\left(1\right)\mid_{U}$
with cosets in $A$ of even and odd polynomial functions on $\mathbb{A}^{n+1}$
respectively. 

$\bullet$ Case $1$ : $n=2m$ is even. Clearly, the $m+1$ odd polynomials
$p_{j}=x_{2j}+ix_{2j+1}$ for $0\leq j\leq m-1$ and $p_{m}=x_{2m}$
have no common zero on $\mathcal{Q}$. Therefore, the corresponding
sections of $\mathcal{O}_{\mathbb{P}^{n}}\left(1\right)\mid_{U}$
generate $M_{n}$ as an $R_{n}$-module. Let us show that $M_{n}$
cannot be generated by less than $m+1$ elements. Otherwise, we could
find in particular $m$ odd polynomial functions $s_{1},\cdots,s_{m}$
on $\mathbb{A}^{n+1}$ with no common zero on $\mathcal{Q}$. Writing
$s_{j}=a_{j}+ib_{j}$ for suitable odd polynomials $a_{j,}b_{j}\in\mathbb{R}\left[x_{0},\ldots,x_{n}\right]$,
this would imply in particular that the $n$ odd real polynomials
$a_{1},\ldots a_{m}$ and $b_{1},\ldots,b_{m}$ have no common zero
on $\mathcal{Q}\cap\mathbb{R}^{n+1}$. This is impossible. Indeed,
since $\mathcal{Q}\cap\mathbb{R}^{n+1}$ is the real $n$-sphere $\mathbb{S}^{n}$,
it follows from Borsuk-Ulam theorem that the map $\phi=\left(a_{1},\ldots,a_{m},b_{1},\ldots,b_{m}\right):\mathbb{S}^{n}\rightarrow\mathbb{R}^{n}$
takes the same value on a pair of antipodal points, hence, being odd,
vanishes on a pair of antipodal points. 

$\bullet$ Case $2$ : $n=2m+1$ is odd. One checks in a similar way
as above that the $m+1$ global sections of $\mathcal{O}_{\mathbb{P}^{1}}\left(1\right)\mid_{U}$
corresponding the odd polynomials $p_{j}=x_{2j}+ix_{2j+1}$, $0\leq j\leq m$
generate $M_{n}$ as an $R_{n}$-module. Now if $M_{n}$ was generated
by $m$ elements, then there would exists $m$ odd polynomials $s_{j}=a_{j}+ib_{j}$
as above for which the polynomials $a_{j},b_{j}\in\mathbb{R}\left[x_{0},\ldots,x_{n}\right]$,
$j=1,\ldots,m$ have no common zero on the real $n$-sphere $\mathbb{S}_{n}$.
But then the continuous map $\phi=\left(a_{1},\ldots,a_{m},b_{1},\ldots,b_{m},0\right):\mathbb{S}^{n}=\mathcal{Q}\cap\mathbb{R}^{n+1}\rightarrow\mathbb{R}^{n}$
would contradict the Borsuk-Ulam theorem. \end{proof}
\begin{rem}
An argument very similar to the one used in the proof above shows that over the subring $\tilde{R}_{n}$ of $\mathbb{R}\left[x_{0},\ldots,x_{n}\right]/\left(x_{0}^{2}+\cdots x_{n}^{2}-1\right)$
consisting of cosets of even polynomials, the module $\tilde{M}_{n}$
consisting of cosets of odd polynomials cannot be generated by less
than $n+1$ elements. This property seems to have been first observed
by Chase (unpublished). Our proof is deeply inspired by an argument
due to Gilmer \cite{Gilmer69} on a slightly different example. 
\end{rem}
\subsection{Further research}
\indent\newline\noindent One may wonder if there exists a complete characterization of obstructions
to infer global tamenes from local tameness for higher dimensional polynomial automorphisms similar to
Theorem \ref{thm:Main}. A good starting point would be to have in general an effective algorithmic way to 
recognize tame automorphisms. Unfortunately, at the present time, such an higher dimensional algorithm 
only exists for automorphisms of a polynomial ring in 3 variables over a field \cite{ShUm}.   

\bibliographystyle{amsplain}

\end{document}